\documentclass[11pt,leqno]{amsart}
\usepackage{amsmath}
\usepackage{amsfonts}
\usepackage{amssymb}
\usepackage{graphicx}
\usepackage{hyperref} 
\usepackage{a4wide} % margins
\usepackage{mathrsfs} % fancy calligraphic style
\usepackage{relsize} % big size of integrals
\usepackage{xcolor} % colors

\usepackage{graphicx}
\usepackage{hyperref}

\usepackage{enumerate}
\usepackage[shortlabels]{enumitem}

\newtheorem{theorem}{Theorem}

\newtheorem{corollary}[theorem]{Corollary}

\newtheorem{definition}[theorem]{Definition}
\newtheorem{example}[theorem]{Example}

\newtheorem{lemma}[theorem]{Lemma}

\newtheorem{proposition}[theorem]{Proposition}
\newtheorem{remark}[theorem]{Remark}

\newcommand{\vol}{\mathrm{vol}}

\newcommand{\dive}{\operatorname{div}}
\renewcommand{\div}{\operatorname{div}}

\newcommand{\Hess}{\operatorname{Hess}}

\newcommand{\rr}{\mathbb{R}}
\renewcommand{\ss}{\mathbb{S}}
\newcommand{\hh}{\mathbb{H}}

\newcommand{\sect}{\operatorname{Sect}}
\newcommand{\ric}{\operatorname{Ric}}

\newcommand{\riem}{\operatorname{Riem}}

\newcommand{\supp}{\operatorname{supp}}

\newcommand{\KN}{\mathbin{\bigcirc\mspace{-15mu}\wedge\mspace{3mu}}}

\newcommand{\inj}{\mathrm{inj}}
\newcommand{\harm}{\mathrm{harm}}

\newcommand{\dist}{\mathrm{dist}}

\renewcommand{\a}{\alpha}

\renewcommand{\b}{\beta}

\renewcommand{\d}{\delta}
\newcommand{\ve}{\varepsilon}

\newcommand{\s}{\sigma}

\newcommand{\vp}{\varphi}

\newcommand{\CL}{\mathcal{L}}

\newcommand{\CO}{\mathcal{O}}
\newcommand{\CQ}{\mathcal{Q}}
\newcommand{\CJ}{\mathcal{J}}

\newcommand{\C}{\mathscr{C}}
\newcommand{\D}{\mathscr{D}}
\newcommand{\E}{\mathscr{E}}

\newcommand{\pr}{\partial_{r}}
\newcommand{\intprod}{\int_{N} \int_{r_{0}}^{+\infty}}

\begin{document}

\title[Unique continuation at infinity: general warped cylinders]{Unique continuation at infinity: \\ Carleman estimates on general warped cylinders}

\begin{abstract}
We obtain a vanishing result for solutions of the inequality $|\Delta u| \leq q_1 |u| + q_2 |\nabla u|$ that decay to zero along a very general warped cylindrical end of a Riemannian manifold. The  appropriate decay condition at infinity on $u$ is  related to the behavior of the potential functions $q_1$ and $q_2$ and to the asymptotic geometry of the end. The main ingredient is a new  Carleman estimate of independent interest. Geometric applications to conformal deformations and to minimal graphs  are presented.
\end{abstract}

\author{Nicol\`o De Ponti}
\address{Universit\'a di Milano-Bicocca, Dipartimento di Matematica e Applicazioni\\ via Roberto Cozzi 55, 20125 Milano, Italy}
\email{nicolo.deponti@unimib.it}
\author{Stefano Pigola}
\address{Universit\'a di Milano-Bicocca, Dipartimento di Matematica e Applicazioni\\ via Roberto Cozzi 55, 20125 Milano, Italy}
\email{stefano.pigola@unimib.it}
\author{Giona Veronelli}
\address{Universit\'a di Milano-Bicocca, Dipartimento di Matematica e Applicazioni\\ via Roberto Cozzi 55, 20125 Milano, Italy} \email{giona.veronelli@unimib.it}
\date{\today}
\maketitle
\tableofcontents

\newpage

%%%%%%%%%%%%%%%%%%%%%%%%%%%%%%%%%%%%%%%%%%%%%%%%%%%%%%%%%%%%%%%%%%%%%%%%%%%%%%%%%%%%5

\section{Introduction}
On a Riemannian manifold $(M,g)$, let us consider a solution $u$ of a PDE of the form
\[
\Delta u + g(X,\nabla u) + qu=0
\]
where $X$ is a  vector field (the drift) and  $q$ is a potential function. If we set $q_1 = |q|$  and $q_2 = |X|$ then, clearly, $u$ satisfies the differential inequality
\begin{equation}\label{eq:ineq}
|\Delta u| \leq q_{1} |u |+ q_{2} |\nabla u|.
\end{equation}
%By \it unique continuation at infinity} for a solution of \eqref{eq:ineq} it is usually meant}
{\it Unique continuation at infinity} for a solution of \eqref{eq:ineq} usually means the property that $u(y)$ vanishes identically on $M$ provided it decays to zero sufficiently fast as $y \to \infty$. 
Such a decay depends deeply on the asymptotic behavior of the potential functions $q_{1},q_{2}$ and on the geometry at infinity of $M$. As soon as the Euclidean space is concerned, the use of a Kelvin transform shows that unique continuation at infinity could be considered as a different manifestation of the classical unique continuation at a fixed point. This viewpoint, however, is not applicable on a general manifold and genuinely different results must be developed, without appealing to any transforms mapping the geometry at infinity to the local geometry around a point.\smallskip
%This viewpoint, however, is completely lacking on a general manifold and genuinely different results, which are  sensitive of the asymptotic geometry of the space, must be developed.}\smallskip

Thus, in unique continuation at infinity problems, one has three players competing to achieve the goal:
\begin{itemize}
 \item The ground space $(M,g)$ with its asymptotic geometry;
 \item The coefficients $q_{1}$ and $q_{2}$ with their asymptotic behavior;
 \item The decay rate at infinity of $u$.
\end{itemize}

To put our paper in perspective let us first assume we are in the classical Euclidean setting $M = \rr^{n}$. Even here, finding the maximal decay rate of nontrivial solutions is a challenging  problem. The sharp rate for  solutions of $|\Delta u|\leq C |u|$, with $C>0$ a constant, is predicted by a conjecture originally due to Landis in the $50$s and later revised by Kenig, \cite{Ke}. Accordingly, $u \equiv 0$ provided $|u(y)| = \CO( e^{- \tau |y|^{1+\ve}})$ for any fixed $\ve>0$ and every $\tau >0$. So far (Kenig version of the) Landis conjecture is settled only in dimensions $n=1,2$, \cite{Ro, LMNN}, whereas, in higher dimensions, the most general result in the literature is due to Meshkov, \cite{Me1},  and requires the strongest decay condition $|u(y)|= \CO ( e^{- \tau |y|^{4/3}})$. What is astonishing in Meshkov paper is that this decay is proved to be sharp when $u$ is a complex-valued function solving $\Delta u = q u$ and $q$ is a bounded complex-valued potential. This shows how subtle the problem is and that the techniques implemented to prove the Landis conjecture in the affirmative must be sensitive of the target field.

Obviously, a different asymptotic behavior of the coefficient $q$ changes radically the situation. For instance, suppose we are considering solutions of \eqref{eq:ineq}   with $q_1(y) = \CO(|y|^{-2})$ and $q_2(y)= \ve |y|^{-1}$ for a sufficiently small $\ve>0$. Then, a classical  and sharp result due independently to Meshkov, \cite{Me2}, and Pan-Wolff, \cite{PW}, states that any solution $u$ decaying faster than polynomially, i.e. $|u(y)| = \CO( |y|^{-\tau})$ for every $\tau >0$, must vanish.

Finally, let us  introduce a different background geometry by assuming that $M=\hh^{n}$ is the standard Hyperbolic space. The influence of the asymptotic geometry of the space is perfectly visible from a well known result by Mazzeo, \cite{Ma}, that compares directly with Landis and Meshkov. Accordingly, if $q_1$ and $q_2$ are bounded functions, the unique continuation at infinity holds for solutions  of \eqref{eq:ineq} decaying more than exponentially, i.e. $|u(y)| = \CO(e^{-\tau |y|})$ for every $\tau >0$. \smallskip

From the pure perspective of PDEs, the main purpose of the present paper is to give a unifying point of view on Meshkov's and Mazzeo's contributions by pointing out, in the general framework of {\it warped cylinders} (see Section \ref{section:notation}), how the above mentioned players (asymptotic geometry of the space, asymptotic behavior of the coefficients and asymptotic decay rate of the solution) interact. Our investigation will lead to a general unique continuation result, Theorem \ref{th-meshkov}, that is specified in a plethora of new conclusions, Section \ref{section:concreteexamples}, and potentially applies to many other situations. Due to its generality and to the technique of proof,  this unique continuation is not able to capture sharp Euclidean behaviors like those related to  Meshkov-Pan-Wolff. This kind of phenomena will be discussed in a forthcoming paper.\smallskip

In the same spirit of \cite{Me1, Ma}, and of a consistent part of the classical literature on the subject, our unique continuation result relies on a new {\it $L^2$ Carleman inequality} that, as it was recently asked in \cite[Question 1.8]{Su2}, incorporates the information on the geometry. Carleman inequalities are weighted elliptic estimates that we write in the general form
\[
\int ( k_{1}\vp^{2} + k_{2}|\nabla \vp|^{2}) e^{h} d\mu_M \leq C \int |\Delta \vp|^{2} e^{h} d\mu_M,
\]
whose validity is originally proved for smooth functions compactly supported in one end of the manifold and then extended globally to functions in the natural weighted Sobolev class; see Section \ref{section:Carleman}. It should be noted that this kind of estimates (in their global form), with a rough constant $C>0$, appear as a part of the weighted Fredholm theory for elliptic operators in asymptotically  Euclidean spaces, \cite{Ba}, and asymptotically Hyperbolic spaces, \cite{Le}. On the other hand, it is by now well understood that the use of Carleman inequalities in unique continuation problems requires a rather explicit dependence of the constant on the weights. In fact, one typically has a family of weights depending on a parameter and the idea is to use the dependence of the constant on this parameter in order to violate the integral inequality on non-vanishing solutions of the PDE.\smallskip

It is worth noting that another consistent part of the literature about unique continuation results is based on different techniques which involve a {\it frequency function}. It is possible that results analogous to ours can also be proved through these  frequency methods. However, these latter methods have been so far poorly developed  in general geometric settings.\smallskip

Unique continuation at infinity plays a central role in several contexts ranging from the absence of eigenvalues embedded in the spectrum of elliptic operators, \cite{Ma, Don}, up to, on a more geometric side, the asymptotic structure of minimal hypersurfaces in both Riemannian and weighted ambient spaces (so to give the citizenship to self-similar solutions of geometric flows), \cite{Wa1, Wa2, De, Su1,Su2}. In this latter context, in order to exemplify the use of our general result, we shall present an essentially sharp asymptotic uniqueness for vertical minimal graphs over Hyperbolic conical ends, Theorem \ref{Th:mimalgraphestimate}. On the same background space, but in a completely different direction, we also point out how one can obtain a-priori asymptotic estimates for the conformal factor in the prescribed scalar curvature equation, Theorem \ref{th:confdef}. These are just instances of the type of results one can obtain via a general unique continuation at infinity on warped cylinders.
\\

\section{Notation}\label{section:notation}

\subsection{Spaces} Throughout this paper, $(M,g)$ will always denote a complete, connected Riemannian manifold without boundary and of dimension $n \geq 2$. Its Riemannian measure is denoted with $\mu_M$. The symbols $\nabla u $, $\Delta u $, $\Hess(u)$ stand for the usual gradient vector field, the (negative definite) Laplacian and the Hessian $(0,2)$-tensor field applied to a function $u$. We also set $\div$ for the divergence operator so that $\Delta u = \div (\nabla u)$. Finally, $\riem$ is reserved for the Riemann $(0,4)$-tensor of $M$ and we use the symbols $\sect(X \wedge Y)$ and $\ric(X,X)$ to denote, respectively, the sectional curvature along the tangent two-plane spanned by $X$ and $Y$, and the Ricci curvature evaluated along the tangent vector field $X$.\smallskip

Given a compact set $K \subset M$, we let $\{\E_{j}\}_{j=1,\cdots,m}$ be the unbounded connected components of $M \setminus K$. They are called the {\it ends} of $M$ with respect to $K$. If $K = \bar \Omega$ for some domain $\Omega \Subset M$ with smooth boundary, then each $\E = \bar \E_{j}$ is a complete, non-compact Riemannian manifold with smooth and compact boundary $\partial \E \not=\emptyset$.\smallskip

We are interested in ends with a special metric structure.

\begin{definition}
 By a warped cylindrical end, or a  warped half-cylinder, we mean any $n$-dimensional warped product manifold of the form
 \[
 \E(r_0) = \left( (r_0,+\infty) \times N ,\, g=dr \otimes dr+ \s(r)^{2} g_{N}\right),\qquad r_0\ge 0,
 \]
 where
\begin{enumerate}[a)]
 \item $(N,g_{N})$ is a complete (often compact), $(n-1)$-dimensional Riemannian manifold without boundary;
 \item $\s : (r_0,+\infty)\to (0,+\infty)$ is a smooth function. 
\end{enumerate}
\end{definition}
When we need to emphasize the dependence on all of the relevant data we shall use the exhaustive notation $\E^{n}_{\s,N}(r_{0})$.

Thus, for instance, if $N = \ss^{n-1}$ and
\begin{equation}\label{def: warpfunc}
\s_{B}(r) =
\begin {cases} 
r & \text{ if } B=0 \\
\frac{1}{\sqrt {-B}}\sinh(\sqrt{-B}\,r) & \text{ if } B<0
\end{cases}
\end{equation}
then, according to the value of $B$, the corresponding warped cylindrical end $\E^{n}_{\s_{B},\ss^{n-1}}(r_{0})$ is nothing but the exterior of a compact geodesic ball of radius $r_{0}>0$ in the spaceforms $\rr^{n}$ or $\hh^{n}_{B}$.\smallskip

Actually, the warping functions $\sigma_0$ and $\sigma_{-1}$ defined in \eqref{def: warpfunc} will play a special role in the paper. It is therefore convenient to give the following.
%Actually, these particular choices of the warping function, but for a generic compact fiber, will play a special role in the paper. It is therefore convenient to give the following.

\begin{definition}
Let $(N,g_{N})$ be a compact Riemannian manifold without boundary and of dimension $\dim N = n-1 \geq 1$.\smallskip

\begin{enumerate}[a)]
    \item By a Euclidean cone $\C^n_{Eu}(N)$ of dimension $n$ and with smooth section $N$, we mean the space $\E_{\s_0,N}^n(0)$. Namely
    \[
    \C^n_{Eu}(N) = ((0,+\infty)\times N , g = dr \otimes dr + r^{2} g_{N}).
    \]
    Correspondingly, any end $\E(r_0) \subset \C^n_{Eu}(N)$ is called a Euclidean conical end.
    \item Similarly, a Hyperbolic cone of dimension $n$ and with smooth section $N$ is the space $\C^n_{Hyp}(N)= \E_{\s_{-1},N}^n(0)$, namely
    \[
    \C^n_{Hyp}(N) = ((0,+\infty)\times N , g = dr \otimes dr + \sinh^{2}(r) g_{N}).
    \]
    Any end $\E(r_0) \subset \C^n_{Hyp}(N)$ is called a Hyperbolic conical end.
\end{enumerate}
\end{definition}
Notice that, in the previous definition, the section $N$ is a closed manifold possibly different from the standard sphere.

\subsection{Functions} Everywhere in the paper the big-O and little-o notation for functions $v$ defined on a warped cylinder $\E$ will always refer to limits as $r\to +\infty$. Accordingly
\[
v(r,x) =  \CO(f(r)) \Longleftrightarrow \limsup_{r\to +\infty}\left|\frac{v(r,x)}{f(r)}\right| < +\infty
\]
and
 \[
 v(r,x) =  o(f(r)) \Longleftrightarrow \lim_{r\to +\infty}\frac{v(r,x)}{f(r)} =0
 \]
uniformly in $x$.

\section{A general unique continuation result}

The following unique continuation result, together with the subsequent Corollary \ref{coro_meshkov}, is the main abstract theorem of this paper. Some concrete incarnations will be presented in Section \ref{section:concreteexamples}.

\begin{theorem}\label{th-meshkov}
Let $(M,g)$ be an $n$-dimensional complete Riemannian manifold with a warped cylindrical end  $\E(r_0) = \E^n_{\s,N}(r_{0})$, $r_0\ge 0$, where $(N,g_N)$ is compact without boundary and the warping function satisfies \begin{equation}\label{sigma lapl}
        |(\log\sigma)'(r)| \le \kappa r, \quad r\gg 1
    \end{equation} 
for some constant $\kappa>0$. Let $q_1,q_2 : M \to [0,+\infty)$ be continuous functions and $u\in C^{2}(M)$ be a solution of the differential inequality
\begin{equation}
|\Delta u| \leq q_1\, |u| + q_2 \, |\nabla u| ,\quad \text{on }M.
\end{equation}
Let $\{h_\tau\}_{\tau>0}$ be a given family of $C^{2}$ non-decreasing functions $(0,+\infty)\to \rr$, and
let $\{k_{1,\tau}\}_{\tau>0}$ and $\{k_{2,\tau}\}_{\tau>0}$ be two families of functions $(0,+\infty)\to [0,+\infty)$. Suppose that there exists a $C^2$ function $G :(0,+\infty) \to \rr$ such that the following conditions are satisfied for every $\tau \gg 1$:
 \begin{equation}\label{a:k2}
0\le k_{2,\tau} \le 2\min\left\{  \left[ 2F_\tau^{2} - 2(n-1) \frac{F_\tau\s'}{ \s} + F_\tau' + F_\tau G' \right]\ ;\ \left[-F_\tau'-F_\tau G'+2\frac{F_\tau\s'}{\s}\right]\right\}     
 \end{equation}
and
\begin{equation}
 0\le  k_{1,\tau} \le 2(A_\tau' + A_\tau G') - k_{2,\tau}F_\tau^2 - G'k_{2,\tau}F_\tau-k_{2,\tau}'F_\tau-k_{2,\tau}F_\tau'\,,    
\end{equation}
where
 \begin{equation}\label{eq1th}
		A_\tau(r) : = \left[F_\tau^{3} - \frac{F_\tau \, (\s^{n-1} \, F_\tau)'}{\s^{n-1}}\right]
	\end{equation}
and
\[
F_\tau:=\frac{1}{2}\left[h_\tau'+(n-1)\frac{\s'}{\s}-G'\right].
\]
Suppose that, for either $\ell=1$ or $\ell=2$,
\begin{equation}\label{e:k infty}
    k_{\ell,\tau} \to + \infty, \text{ as }\tau \to +\infty
\end{equation}
and that, for $\ell=1,2$,
\begin{equation}\label{ass: kq}
   \sup_{\E(r_0)} \frac{q_{\ell}^2}{k_{\ell,\tau}} \to 0 \text{ as }\tau \to + \infty.
\end{equation}
If, for every $\tau \gg 1$,
\begin{equation}\label{eq-decay}
\int_{\E(r_0)} u^{2} \, (1+k_{2,\tau}) \, e^{h_{\tau}} d\mu_M <+\infty
\end{equation}
and
\begin{equation}\label{eq-decay''}
    \int_{\E(r)\setminus \E(2r)}|\nabla u|^2 e^{h_\tau} d\mu_M = o(r^2),
\end{equation}
then $u \equiv 0$ on $M$.
\end{theorem}
\begin{remark}\label{rem: curvcond}
    As it will be visible from the proof,  condition \eqref{sigma lapl} is only used to ensure the existence of a family of Laplacian cut-off functions; see \eqref{eq-cutoff}. According to \cite{BS}, the existence of such a family is guaranteed 
    when the curvature condition 
        \begin{equation}\label{eq9th}
 \ric(r,x) \geq - C  \left(1 + r^{2}\right)
\end{equation}
is satisfied for every $(r,x) \in \E(r_0)$ and for some constant $C>0$.
While the result from \cite{BS} holds in the general setting of complete manifolds without requiring any symmetry, in the special case of warped cylindrical ends the assumption on Ricci can be weakened and replaced by \eqref{sigma lapl}.
\end{remark}
Under stronger (yet very general) assumptions on $h_{\tau}$ and $q_1,q_2$,
the unpleasant growth condition 
    \eqref{eq-decay''} on the gradient of $u$ can be removed. 
\begin{corollary}\label{coro_meshkov}
Suppose that for every $\tau>0$,
\begin{equation}\label{mainass: cor}
\sup_{(s,x) \in \E(r)\setminus \E(8r)} ( r|h_\tau'(s)| + q_1(s,x) + r\,q_2(s,x)) =\CO(r^2).
\end{equation}
Then  Theorem \ref{th-meshkov} holds true without requiring 
 assumption \eqref{eq-decay''}.
\end{corollary}

\section{Carleman estimates}\label{section:Carleman}

A crucial ingredient in the proof of Theorem \ref{th-meshkov} is represented by a suitable global Carleman estimate which is sensitive of the geometry of the space. We shall achieve the desired estimate in two steps. First, we consider compactly supported functions and next we extend the estimate to general functions by using an approximation argument that relies on the existence of Laplacian cut-off functions.

\subsection{Compactly supported functions}

The first important step is represented by the general Carleman-type estimate contained in the next Lemma.
\begin{lemma}\label{lemma1grad}
	Let $\E(r_0)=\E^{n}_{\s,N}(r_0)$, $r_0\ge 0$, be a warped cylindrical end. 
	Let $h,G :(0,+\infty) \to \rr$ be two given $C^{2}$ functions. Then, for every $v \in C^{\infty}_{c}(\E(r_0))$ the following Carleman-type estimate holds
	\begin{equation}\label{eq5}
		\int_{\E(r_0)}  v^{2} \, k_1 \,e^{h} d\mu_M + \int_{\E(r_0)}  |\nabla v|^{2} \, k_2 \,e^{h} d\mu_M \leq \int_{\E(r_0)}(\Delta v)^{2} \, e^{h} d\mu_M,
	\end{equation}
	where $k_1$ and $k_2$ are such that
 \[
k_2 \le 2\min\left\{  \left[ 2F^{2} - 2(n-1) \frac{F\s'}{ \s} + F' + FG' \right]\ ;\ \left[-F'-FG'+2\frac{F\s'}{\s}\right]\right\}
\]
and
\[
  k_1 \le 2(A' + AG') - k_2F^2 - G'k_2F-k_2'F-k_2F'\,, 
\]
	with \begin{equation}\label{eq1}
		A(r) : = \left[F^{3} - \frac{F \, (\s^{n-1} \, F)'}{\s^{n-1}}\right]
	\end{equation}
and
\[
F:=\frac{1}{2}\left[h'+(n-1)\frac{\s'}{\s}-G'\right].
\]
\end{lemma}
\begin{remark}\label{rem: mazzeo}
    Let $\mathrm W(\E(r_0))$ be the closure of $C^\infty_c(\E(r_0))$ with respect to the weighted norm
    \[
\|u\|_{\mathrm W}:=\|u\|_{L^2(\E(r_0),k_1\,e^h\,d\mu_M)}+\||\nabla u|\|_{L^2(\E(r_0),k_2\,e^h\,d\mu_M)}+\|\Delta u\|_{L^2(\E(r_0),e^h\,d\mu_M)}.
    \]
    Then, if $k_1$ and $k_2$ are non-negative the validity of \eqref{eq5} trivially extends to all $v\in\mathrm W(\E(r_0))$.
\end{remark}
\begin{proof}
 Fix $v \in C^{\infty}_{c}(\E(r_0))$. Define $f(r)=\int_{r_0}^r F(t)\,dt$ and $g=e^G$. Let
 \[
 w := e^{f}v\in C^{\infty}_{c}(\E(r_0)).
 \]
 Then,
 \[
 e^{f}\Delta v = \bar \Delta w,
 \]
 where we have set
 \[
 \bar \Delta (\bullet) := e^{f}\Delta (e^{-f }\bullet).
 \]
 With this notation, noticing that %$g=e^{h-2f}\sigma^{n-1}$
 
 $$g(r)=g(r_0)\frac{e^{h(r)-2f(r)}\sigma^{n-1}(r)}{e^{h(r_0)}\sigma^{n-1}(r_0)},$$ the desired inequality \eqref{eq5} takes the form
\begin{align}\label{eq5a0}
&\intprod  (\bar \Delta w)^{2}g \, dr d\mu_{N}\\
&\geq \intprod w^{2} g  \, k_1 \,dr d\mu_{N} + \intprod g\,k_2 \left\{ |\nabla w|^2+w^2(f')^2-2f'w\partial_rw\right\}\, dr d\mu_{N}\nonumber\\
&=\intprod w^{2}\left( g  \, k_1+g\,k_2F^2+(g\,k_2\,F)'\right) \,dr d\mu_{N} + \intprod g\,k_2 \left((\partial_rw)^2+\sigma^{-2}|\nabla^Nw|^2\right)\, dr d\mu_{N},\nonumber
\end{align}
where in the last equality we have used the expression of the gradient in polar coordinates and the integration by parts formula
\begin{equation}\label{eq:intbyparts}
\int_{r_{0}}^{r_{1}} \a \b \b' dt = -\frac{1}{2} \int_{r_{0}}^{r_{1}} \a' \b^2 dt,    
\end{equation}
which is obviously valid provided $\b^{(i)}(r_{0}) = \b^{(i)}(r_{1}) = 0$, $i=0,1$.

A direct computation that uses the standard formula
\[
\Delta w = \pr^{2}w + (n-1)\frac{\s'}{\s} \pr w + \frac{1}{\s^{2}} \Delta_{N} w
\]
gives
\begin{align*}
 \bar \Delta w &= \pr ^{2} w + (n-1) \frac{\s'}{\s} \pr w + \frac{1}{\s^{2}}\Delta_{N}w+ \left(F^{2} - \frac{(F\, \s^{n-1})'}{\s^{n-1}} \right) \, w  -2 F \, \pr w.
\end{align*}
We split the RHS of this latter as
\[
\bar \Delta w = \CQ w - \CL w
\]
where
\[
\CQ w = \pr ^{2} w + (n-1) \frac{\s'}{\s} \, \pr w + \frac{1}{\s^{2}} \, \Delta_{N}w + \left(F^{2} - \frac{(F \, \s^{n-1})'}{\s^{n-1}} \right) \, w
\]
and
\[
\CL w = 2F \, \pr w.
\]
Then,
\begin{align*}
 (\bar \Delta w)^{2} &\geq (\CL w)^{2} - 2\CL w \, \CQ w.
\end{align*}
By expanding the above expressions we obtain
\begin{equation}\label{eq6grad}
 \intprod (\bar \Delta w)^{2}  g \, dr d\mu_{N} \geq \CJ_{1} + \CJ_{2} + \CJ_{3},
\end{equation}
where 
\[
 \CJ_{1} :=  \intprod g\left\{ 4 F^{2}  (\pr w)^{2} - 4F  \pr w  \pr^{2}w - 		4(n-1) F  \frac{\s'}{\s}(\pr w)^{2} \right\} dr d\mu_{N},
 \]
 \[
 \CJ_{2} :=  \intprod  g\left\{ -4 \frac{F}{\s^{2}}  \pr w  \Delta_{N} w\right\} drd\mu_{N},
 \]
 and, recalling \eqref{eq1},
 \begin{align*}
 \CJ_{3} &:=  \intprod g\left\{ -4 F  w  \pr w   \left( F^{2} - \frac{(F  \s^{n-1})'}{\s^{n-1}}\right)\right\} drd\mu_{N}\\
 &= \int_N \int_{r_0}^{+\infty} -4 g \, w \, \partial_r w \, A \, dr d\mu_N.
 \end{align*}

In the sequel we will repeatedly use the integration by parts formula \eqref{eq:intbyparts}.
 Let us first consider the term $\CJ_{1}$. Let us compute:
\[
\intprod Fg \, (\pr w \, \pr^{2}w) dr d\mu_{N}  = -\frac{1}{2} \intprod (Fg)'(\pr w)^{2} dr d\mu_{N}.
\]
Therefore
\begin{align*}
 \CJ_{1} &= 2 \intprod (\pr w)^{2} \left\{ \left(2F^{2} - 2(n-1) F \frac{\s'}{\s}\right) g +(F g)' \right\} dr d\mu_{N}\\
 &= \intprod (\pr w)^{2} F^{2}g \left \{ 4 - 4(n-1) \frac{\s'}{F \s} +2 \frac{F'}{F^{2}} + \frac{2}{F} \frac{g'}{g} \right\} dr d\mu_{N}.
\end{align*}
Now we consider $\CJ_{2}$. Integrating by parts in $\rr$ we get
\begin{align*}
 \CJ_{2} = \CJ_{2}^{a} + \CJ_{2}^{b}
\end{align*}
where 
\begin{align*}
 \CJ_{2}^{a}  &:= 4 \intprod w \, \frac{Fg}{\s^{2}} \, \pr (\Delta_{N}w) \, drd\mu_{N}\\
 \CJ_{2}^{b} &:= 4 \int_{r_{0}}^{+\infty} \left( \frac{Fg}{\s^{2}}\right)' \int_{N} w \, (\Delta_{N}w) \,d\mu_{N} dr.
\end{align*}
We are going to elaborate each of these terms separately. Integrating by parts in $N$ we compute
\begin{equation}\label{eq-j2bgrad}
\CJ_{2}^{b} = -4 \int_{r_{0}}^{+\infty} \left( \frac{Fg}{\s^{2}} \right)' \int_{N} |\nabla^{N}w|^{2}d\mu_{N}dr.
\end{equation}
 On the other hand, using the fact that
\begin{equation}\label{eq7grad}
[\Delta_{N} , \pr] = 0,
\end{equation}
and integrating again by parts in $N$, we obtain
\begin{align*}
 \CJ_{2}^{a} &= 4 \int_{r_{0}}^{+\infty} \frac{Fg}{\s^{2}} \int_{N} w \Delta_{N}(\pr w) \, d\mu_{N}dr\\
 &= 4 \int_{r_{0}}^{+\infty} \frac{Fg}{\s^{2}} \int_{N} \Delta_{N} w \, \pr w \, d\mu_{N}dr\\
 &= - \CJ_{2}.
 \end{align*}
 Therefore
 \[
 \CJ_{2} = \frac{1}{2} \CJ_{2}^{b} = -2 \int_{r_{0}}^{+\infty} \left( \frac{Fg}{\s^{2}} \right)' \int_{N} |\nabla^{N}w|^{2}d\mu_{N}dr.
 \]
 Finally, we consider $\CJ_{3}$. Integrating by parts once more in $\rr$ we have
 \begin{align*}
 \CJ_{3} &= 2 \intprod \left\{ g \, \left( F^{3} - F \frac{(F \s^{n-1})'}{\s^{n-1}}\right) \right\}' w^{2} dr d\mu_{N}\\
 &= 2 \intprod (g A)'  w^{2} \, dr d\mu_{N}.
 \end{align*}
All together we have obtained that
\begin{align*}
 &\intprod (\bar \Delta w)^{2}  g \, dr d\mu_{N} \geq \CJ_{1} + \CJ_{2} + \CJ_{3}\\
 &= \intprod (\pr w)^{2} F^{2}g \left \{ 4 - 4(n-1) \frac{\s'}{F \s} +2 \frac{F'}{F^{2}} + \frac{2}{F} \frac{g'}{g} \right\} dr d\mu_{N}\\
& -2  \int_{r_{0}}^{+\infty} \left( \frac{Fg}{\s^{2}} \right)' \int_{N} |\nabla^{N}w|^{2}d\mu_{N}dr +2 \intprod (g A)'  w^{2} \, dr d\mu_{N}.
\end{align*}
By the choice of $k_1$ and $k_2$, this latter implies 
\eqref{eq5a0} and concludes the proof.
\end{proof}

\subsection{Functions with non-compact support}

Assuming that $u$ and $\nabla u$ have the right integrability properties, we are able to use density arguments to extend the validity of the previous Carleman-type estimate beyond the class of compactly supported functions. This is the content of the next result.

\begin{proposition}
    \label{coro6}
    
    Let $(M,g)$ be an $n$-dimensional complete Riemannian manifold with a warped cylindrical end  $\E(r_0) = \E^n_{\s,N}(r_{0})$, $r_0\ge 0$, where $(N,g_N)$ is compact without boundary and the warping function satisfies \begin{equation}\label{eqprop:sigma lapl}
        |(\log\sigma)'(r)| \le \kappa r, \quad r\gg 1
    \end{equation} 
    for some constant $\kappa>0$.
 Assume that the following Carleman-type estimate
\begin{equation}\label{eq12}
\int_{\E(r_0)} v^{2} \, k_1\, e^hd\mu_M + \int_{\E(r_0)} |\nabla v|^{2} \, k_2\, e^hd\mu_M \leq \int_{\E(r_0)}(\Delta v)^{2} \, e^hd\mu_M
\end{equation}
holds for every $v \in C^{\infty}_{c}( \E(r_0))$ and for some density functions $h : \E(r_0) \to \rr$ and  $k_1,k_2: \E(r_0) \to (0,+\infty)$.
Then the Carleman-type estimate
 \begin{equation}\label{eq13}
\int_{\E(r_0)}  w^{2}\, k_1 \, e^hd\mu_M +\int_{\E(r_0)}  |\nabla w|^{2}\, k_2 \, e^hd\mu_M \leq  \Lambda \int_{\E(r_0)}(\Delta w)^{2} \,e^hd\mu_M,
\end{equation}
holds for some universal constant $\Lambda=\Lambda(\kappa,n)> 0$ and for every $w \in L^{2} (\E(r_0), (1+k_2)\,e^hd\mu_M)$ satisfying $\supp w \subset \E(r_0)$ and $\||\nabla w|\|_{L^2(\E(r)\setminus \E(2r),e^hd\mu_M)}=o(r)$.
\end{proposition}

\begin{proof}
    By the assumption \eqref{eqprop:sigma lapl}, there exists a family of smooth cut-off functions $\vp=\vp_{R} : \E(r_0) \to [0,1]$, $R>r_0$, satisfying the following conditions:
\begin{equation}\label{eq-cutoff}
\begin{array}{llll}
  (i) & \vp =1  \text{ on } \E(r_0) \setminus \E(R) \\
  (ii) & \supp \vp \subset \E(r_0) \setminus \E(2R)\\
  (iii) & \| \nabla \vp \|_{\infty} \leq {C}/{R}\\
  (iv) & \| \Delta \vp \|_{\infty} \leq C
\end{array}
\end{equation}
for some $C>0$ which depends on $n$ and $\kappa$;
    to this end, one can first choose $\Phi(r)$ so that $\Phi\in C^\infty([0,+\infty),[0,1])$, $\Phi\equiv 1$ in $[0,1]$ and $\Phi\equiv 0$ in $[2,+\infty)$;
    then, for $(r,x)\in\E(r_0)$, define $\vp_R(r,x)=\Phi(r/R)$ and note that 
    \[
\Delta \vp_R = \partial^2_{rr}\vp_R + (n-1)\frac{\sigma'}{\sigma}\partial_{r}\vp_R=\frac{1}{R^2}\Phi''(r/R) + (n-1)\frac{\sigma'}{\sigma}\frac{1}{R}\Phi'(r/R).
    \]    
Let $w \in L^{2}(\E(r_{0}), (1+k_2)\,e^hd\mu_M)$ satisfying $\||\nabla w|\|_{L^2(\E(r) \setminus \E(2r),e^hd\mu_M)}=o(r)$. 
First, notice that for any $\varepsilon>0$, by the Young's and Cauchy-Schwarz inequalities, it holds 
\begin{align*}
   &(1-\varepsilon) \int_{\E(r_0) \setminus \E(R)}  \vp^2|\nabla w|^{2} \, k_2 \,  e^hd\mu_M\\
   &\le \int_{\E(r_0) \setminus \E(R)}  |\nabla (\vp w)|^{2} \, k_2 \,  e^hd\mu_M + (\varepsilon^{-1}-1)\int_{\E(r_0) \setminus \E(R)}  w^2|\nabla \vp|^{2} \, k_2 \,  e^hd\mu_M.
\end{align*}
Then, from \eqref{eq12} we get
\begin{align}\label{eq13b}
 &\int_{\E(r_0) \setminus \E(R)}  w^{2} \, k_1 \,  e^hd\mu_M+(1-\varepsilon)\int_{\E(r_0) \setminus \E(R)}  |\nabla w|^{2} \, k_2 \,  e^hd\mu_M \\
 &\leq \int_{\E(r_{0})} (w\vp)^{2} \, k_1 \,  e^hd\mu_M +\int_{\E(r_{0})} |\nabla(w\vp)|^{2} \, k_2 \,  e^hd\mu_M + (\varepsilon^{-1}-1)\int_{\E(r_0) \setminus \E(R)}  w^2|\nabla \vp|^{2} \, k_2 \,  e^hd\mu_M \nonumber\\
 &
 \leq \int_{\E(r_{0})}(\Delta (w\vp))^{2} \, e^hd\mu_M+(\varepsilon^{-1}-1)\int_{\E(r_0) \setminus \E(R)}  w^2|\nabla \vp|^{2} \, k_2 \,  e^hd\mu_M.\nonumber
 \end{align}
 Now, using again the  Cauchy-Schwarz and Young's inequalities we obtain
 \begin{align*}
\int_{\E(r_{0})}( \Delta(w \vp))^{2} e^hd\mu_M &\leq  3\int_{\E(r_{0})} w^{2} (\Delta \vp)^{2} e^hd\mu_M + 12\int_{\E(r_{0})} |\nabla w|^{2} |\nabla \vp|^{2} e^hd\mu_M
\\
& 
+ 3 \int_{\E(r_{0})} (\Delta w)^{2} \vp^{2} e^hd\mu_M\\
&\leq 3  \| \Delta \vp\|^{2}_{\infty} \int_{\E(R) \setminus \E(2R)}w^{2} e^hd\mu_M + 12 \| \nabla \vp\|^2_{\infty} \int_{\E(R) \setminus \E(2R)} |\nabla w|^{2} e^hd\mu_M\\
& + 3 \int_{\E(r_0) \setminus \E(2R)} (\Delta w)^{2} e^hd\mu_M.
 \end{align*}
 Therefore, there exists a constant $\Lambda =\Lambda(n,\kappa)>0$ such that
 \begin{align}\label{eq13c}
 \int_{\E(r_{0})}(\Delta(w\vp))^{2}\, e^hd\mu_M  &\leq \Lambda\int_{\E(R) \setminus \E(2R)} w^{2} e^hd\mu_M + \frac{\Lambda}{R^{2} }\int_{\E(R) \setminus \E(2R)} |\nabla w |^{2}e^hd\mu_M\\ \nonumber
 &+\Lambda \int_{\E(r_0) \setminus \E(2R)} (\Delta w)^{2} e^hd\mu_M. \nonumber
 \end{align}
 
Note that, since $w \in L^{2} (\E(r_{0}), (1+k_2)\,e^hd\mu_M)$ and $\||\nabla w|\|_{L^2(\E(R) \setminus \E(2R),e^hd\mu_M)}=o(R)$, the first two integrals on the RHS vanish as $R \to +\infty$ . 
Similarly
\begin{equation}\label{eq13c'} \int_{\E(r_0) \setminus \E(R)}  w^2|\nabla \vp|^{2} \, k_2 \,  e^hd\mu_M \le \frac{\Lambda}{R^2}
\int_{\E(r_0) \setminus \E(R)}  w^2 \, k_2 \,  e^hd\mu_M \to 0,\qquad \text{ as } R\to +\infty.
\end{equation}
Therefore, to conclude, we insert \eqref{eq13c} and \eqref{eq13c'} into \eqref{eq13b} and we let $R \to +\infty$. By applying the monotone convergence theorem we obtain
\[
 \int_{\E(r_{0})}  w^{2} \, k_1 \,  e^hd\mu_M+(1-\varepsilon)\int_{\E(r_{0})}  |\nabla w|^{2} \, k_2 \,  e^hd\mu_M \le \Lambda \int_{\E(r_{0})} (\Delta w)^2\,e^hd\mu_M.
\]
As $\varepsilon>0$ can be chosen arbitrarily small, this permits to conclude.
\end{proof}

\section{Proof of the unique continuation result}
We are now in the position to give the proof of our main result.
\begin{proof}[Proof of Theorem \ref{th-meshkov}]
Let $\xi : M  \to [0,1]$ be any $C^{2}$ function such that $\xi=0$ on $M \setminus \E(r_0)$ and $\xi=1$ on $\E(r_{0}+1)$. For the moment, let $\tau>0$ be arbitrary. By defining $w=\xi \, u$ we have that $w\in  L^{2}(M,(1+k_2)\,e^{h_{\tau}}d\mu_M)$ with $\supp w \subset \E(r_0)$, and 
$\||\nabla w|\|_{L^2(\E(r) \setminus \E(2r),e^{h_{\tau}}d\mu_M)}=o(r)$. By Lemma \ref{lemma1grad} and
Proposition \ref{coro6}, we get the Carleman estimate for $w$:
\begin{equation}\label{eq12'}
\int_{\E(r_0)} u^{2}\xi^2\, k_{1,\tau}\, e^{h_{\tau}}d\mu_M + \int_{\E(r_0)} |\nabla (u\xi)|^2\, k_{2,\tau}\, e^{h_{\tau}}d\mu_M\leq  \Lambda \int_{\E(r_0)}(\Delta (u\xi))^{2} \, e^{h_{\tau}}d\mu_M.
\end{equation}
Let us compute
\begin{align*}
|\Delta (u\xi)|&=|u\Delta \xi + 2 \left\langle\nabla u,\nabla\xi\right\rangle + \xi\Delta u| \\
&\le  |u\Delta \xi + 2 \left\langle\nabla u,\nabla\xi\right\rangle| + \xi q_{1} |u| + \xi q_{2} |\nabla u|,
\end{align*}
and
\begin{align*}
|\Delta (u\xi)|^2&\leq
3|u\Delta \xi + 2 \left\langle\nabla u,\nabla\xi\right\rangle|^2 + 3 q_{1}^2\xi^2u^2+ 3 q_{2}^2\xi^2|\nabla u|^2.
\end{align*}
Inserting in \eqref{eq12'} gives
\begin{align*}
\int_{\E(r_0+1)} u^{2}\xi^2\, k_{1,\tau}\, e^{h_{\tau}}d\mu_M + \int_{\E(r_0+1)} |\nabla (u\xi)|^2\, k_{2,\tau}\, e^{h_{\tau}}d\mu_M&\leq  3\Lambda \int_{\E(r_0)}\left|u\Delta \xi + 2 \left\langle\nabla u,\nabla\xi\right\rangle\right|^2e^{h_{\tau}}d\mu_M \\
& +
3\Lambda \int_{\E(r_0)}
\left[q_{1}^2\xi^2u^2+  q_{2}^2\xi^2|\nabla u|^2
\right] \, e^{h_{\tau}}d\mu_M,
\end{align*}
so that
\begin{align*}
&\int_{\E(r_{0}+1)} \{k_{1,\tau}-3\Lambda  q_{1}^2\}u^{2} \, e^{h_{\tau}}d\mu_M + 
\int_{\E(r_{0}+1)} \{ k_{2,\tau}-3\Lambda  q_{2}^2\}|\nabla u|^{2} \, e^{h_{\tau}}d\mu_M \\
&\leq 3\Lambda \int_{\E(r_0)\setminus \E(r_0+1)}|u\Delta \xi + 2 \left\langle\nabla u,\nabla\xi\right\rangle|^2  \, e^{h_{\tau}}d\mu_M
+ 3\Lambda \int_{\E(r_0)\setminus \E(r_0+1)}
\left[q_{1}^2\xi^2u^2+  q_{2}^2\xi^2|\nabla u|^2
\right] \, e^{h_{\tau}}d\mu_M.
\end{align*}
Since $h_\tau$, hence $e^{h_\tau}$, are non-decreasing, dividing both sides of the latter by $e^{h_\tau(r_0+1)}$ we obtain
\begin{align*}
&\int_{\E(r_{0}+1)} \{k_{1,\tau}-3\Lambda  q_{1}^2\}u^{2} \, d\mu_M + 
\int_{\E(r_{0}+1)} \{ k_{2,\tau}-3\Lambda  q_{2}^2\}|\nabla u|^{2} \, d\mu_M \\
&\leq 3\Lambda \int_{\E(r_0)\setminus \E(r_0+1)}|u\Delta \xi + 2 \left\langle\nabla u,\nabla\xi\right\rangle|^2  \, d\mu_M 
+ 3\Lambda \int_{\E(r_0)\setminus \E(r_0+1)}
\left[q_{1}^2\xi^2u^2+  q_{2}^2\xi^2|\nabla u|^2
\right] \, d\mu_M.
\end{align*}
Suppose that \eqref{e:k infty} is satisfied if $\ell=1$, the other case being similar. Then
\begin{align*}
&\int_{\E(r_{0}+1)}    \left\{k_{1,\tau}\left(1-3\Lambda  \frac{q_{1}^2}{k_{1,\tau}}\right)\right\}u^{2} \, d\mu_M  \\
&\leq 3\Lambda \int_{\E(r_0)\setminus \E(r_0+1)}|u\Delta \xi + 2 \left\langle\nabla u,\nabla\xi\right\rangle|^2  \, d\mu_M 
+ 3\Lambda \int_{\E(r_0)\setminus \E(r_0+1)}
\left[q_{1}^2\xi^2u^2+  q_{2}^2\xi^2|\nabla u|^2
\right] \, d\mu_M.
\end{align*}
Since the RHS integral is finite and independent of $\tau$, assumptions \eqref{e:k infty} and \eqref{ass: kq} permit to apply Fatou's Lemma as $\tau\to\infty$ and deduce that $u \equiv 0$ on $\E(r_0+1)$. By standard local unique continuation results, \cite{Ka}, we conclude that $u\equiv 0$ on $M$.
\end{proof}

In the case where $h'_{\tau}$, $q_1$ and $q_2$ have controlled growth, the assumptions of Theorem \ref{th-meshkov} can be weakened by removing the integrability assumption on $|\nabla u|$.

\begin{proof}[Proof (of Corollary \ref{coro_meshkov})]
 Let $\vp_R$ be a family of cut-offs as in the proof of Proposition \ref{coro6} and define a new family of smooth cut-offs as $\psi=\psi_R=(1-\vp_R)\, \vp_{4R}$, so that $\psi_R$ is supported in $\E(R)\setminus \E(8R)$ and $\psi_R=1$ on $\E(2R)\setminus \E(4R)$. Integrating by parts we get
\begin{align*}
 \int_{\E(r_0)} \psi^2 |\nabla u|^{2} \, e^{h_{\tau}}d\mu_M = &- \int_{\E(r_0)} \psi^2 u \Delta u \, e^{h_{\tau}}d\mu_M\\ \nonumber
 &-  \int_{\E(r_0)} \psi^2 u \langle \nabla u ,\nabla h_\tau \rangle \, e^{h_{\tau}}d\mu_M - \int_{\E(r_0)} 2\psi u \langle \nabla u,\nabla \psi \rangle \, e^{h_{\tau}}d\mu_M. \nonumber
\end{align*}  
We apply Cauchy-Schwarz and Young's inequality to each of the RHS integrals, obtaining for any $\varepsilon>0$,
\begin{align*}
 \int_{\E(r_0)} \psi^2 |\nabla u|^{2} \, e^{h_{\tau}}d\mu_M 
 &\le  4\varepsilon^{2}\int_{\E(r_0)} \psi^2 u^2  \, e^{h_{\tau}}d\mu_M + \varepsilon^{-2} \int_{\E(r_0)} \psi^2  (\Delta u)^2 \, e^{h_{\tau}}d\mu_M\\ &+ \int_{\E(r_0)} \psi^2 u^2 |h_\tau'|^2 \, e^{h_{\tau}}d\mu_M
 + \frac{1}{4}\int_{\E(r_0)} \psi^2  | \nabla u|^2  \, e^{h_{\tau}}d\mu_M\\
 &+ 4 \int_{\E(r_0)} u^2 |\nabla \psi|^2 \, e^{h_{\tau}}d\mu_M
+ \frac{1}{4}\int_{\E(r_0)} \psi^2 | \nabla u|^2 \, e^{h_{\tau}}d\mu_M. 
\end{align*} 
Inserting
\[
(\Delta u)^2 \le 2 q_1^2 u^2 + 2q_2^2|\nabla u|^2
\]
gives, with $Q_{2,r}:=\sup_{\E(r)\setminus \E(8r)} q_2$, 
\begin{align*}
 \frac{1}{2}\int_{\E(r_0)} \psi^2 |\nabla u|^{2} \, e^{h_{\tau}}d\mu_M 
 &\le  \left(4\varepsilon^{2}+\sup_{\operatorname{supp}\psi}
 [2\varepsilon^{-2} q_1^2+|h_\tau'|^2]\right)\int_{\E(r_0)} \psi^2 u^2  \, e^{h_{\tau}}d\mu_M\\
&+ 4 \int_{\E(r_0)} u^2 |\nabla \psi|^2 \, e^{h_{\tau}}d\mu_M + 2\ve^{-2} Q_{2,r}^2\int_{\E(r_0)} \psi^2 |\nabla u|^{2} \, e^{h_{\tau}}d\mu_M.
\end{align*} 
Therefore, as soon as 
$\varepsilon^{-2} Q^2_{2,R}<1/8$, we get
\begin{align*}
 \frac{1}{4}\int_{\E(r_0)} \psi^2 |\nabla u|^{2} \, e^{h_{\tau}}d\mu_M 
 &\le  \left(4\varepsilon^{2}+\sup_{\operatorname{supp}\psi}
 [2\varepsilon^{-2} q_1^2+|h_\tau'|^2]\right)\int_{\E(r_0)} \psi^2 u^2  \, e^{h_{\tau}}d\mu_M\\
&+ 4 \int_{\E(r_0)} u^2 |\nabla \psi|^2 \, e^{h_{\tau}}d\mu_M\,.
\end{align*} 
By taking advantage of the properties \eqref{eq-cutoff} and dividing by $R^2$ we obtain
\begin{align*}
&\frac{1}{4R^2}\int_{\E(R)\setminus \E(8R)} |\nabla u|^{2} \, e^{h_{\tau}}d\mu_M\\ 
 &\le  \left(\frac{4}{R^2}\varepsilon^{2}+\frac{1}{R^2}\sup_{\operatorname{supp}\psi}
 [2\varepsilon^{-2} q_1^2+|h_\tau'|^2]\right)\int_{\E(R) \setminus \E(8R)}u^2  \, e^{h_{\tau}}d\mu_M 
+ \frac{\tilde{C}}{R^4} \int_{\E(r_0)} u^2  \, e^{h_{\tau}}d\mu_M\,,
\end{align*} 
for a constant $\tilde{C}$ depending only on $n$ and $\kappa$.
For $r\ge r_0$, define 
\[
Q_{1,r}:=\sup_{(s,x) \in \E(r)\setminus \E(8r)} ( |h_\tau'(s)| + q_1(s,x)/r).
\]
By \eqref{mainass: cor}, $Q_{1,r}=\CO(r)$ and $Q_{2,r}^2=\CO(r^2)$. 
Choose $\ve=\ve_R:=\max\{R,3\, Q_{2,R}\}$. Then $\varepsilon_R^{-2}Q^2_{2,R}<1/8$ and $\ve_R^{2}/R^2=\CO(1)$, so that 
\[
\frac{1}{R^2}\sup_{\operatorname{supp}\vp}
 [2\varepsilon_R^{-2} q_1^2+|h_\tau'|^2]\leq 2\ve_R^{-2} Q_{1,R}^2 + \frac{Q_{1,R}^2}{R^2} = \CO(1).
\]
Since $u\in L^2(\E(r_0),e^{h_{\tau}}d\mu_M)$ thanks to assumption \eqref{eq-decay}, by taking the limit as $R\to +\infty$ we get the validity of \eqref{eq-decay''} and thus the proof is concluded.
\end{proof}

\section{Some concrete examples}\label{section:concreteexamples}

The aim of this section is to show how much flexible the unique continuation result in Theorem \ref{th-meshkov} is by presenting a number of concrete situations where it applies. This will enable us to recover and extend known results in the Euclidean and Hyperbolic settings and to exemplify situations where wilder geometries at infinity appear. All of this, clearly, boils down to a suitable choice of the parameters and functions involved. As usual in the literature on the subject, in the statements below we will assume point-wise decays of the solution $u$ that imply the integral condition \eqref{eq-decay}.\smallskip

\subsection{Euclidean conical ends} 

On Euclidean conical ends, we point out the following unique continuation results.

\begin{corollary}\label{cor: concEucl}
Let $\E \subset \C^n_{Eu}(N)$ be a Euclidean conical end and let $u \in C^{2}(\E)$ be a solution of
\[
|\Delta u | \leq q_{1} |u| + q_{2} |\nabla u|
\]
for some $0\leq q_{1},q_{2} \in C^{0}(\E)$. Let $C_1,C_2$ be non-negative constants. Then $u \equiv 0$ in each of the following situations.\smallskip
\begin{enumerate}[a)]
 \item $q_{1}(r,x) =  C_{1} r^{-(2-\frac{3}{2}\b)}$, $q_{2}(r,x)=C_{2} r^{-(1-\frac{1}{2}\beta)}$, $0<\b \le 2$, and\smallskip
 \[
|u(r,x)| = \CO( e^{-\tau r^{\beta}}),\quad \forall\, \tau \gg 1.\smallskip
 \]
 \item $q_1(r,x) = C_{1} r^{-2}(\log{r})^{-(2-\frac{3}{2}\gamma)}$, $q_2(r,x) = C_{2}r^{-1}(\log{r})^{-(1-\frac{1}{2}\gamma)}$, $\gamma >1$,   and\smallskip
 \[
|u(r,x)| = \CO \left( e^{-\tau (\log{r})^{\gamma}} \right), \quad \forall \tau \gg 1.
 \]
\item $q_{1}(r,x) =  C_{1} r^{-(2-\frac{3}{2}\b)}$, $q_{2}(r,x)=C_{2} r^{-(1-\frac{1}{2}\beta)}$, $\b > 2$,\smallskip
 \[
|u(r,x)| = \CO( e^{-\tau r^{\beta}}),\quad \forall\, \tau \gg 1,\smallskip
 \]
 and
\begin{equation}\label{ass: notcoreucl}
\int_{\E(r_0)\setminus \E(2r)}|\nabla u|^2e^{2\tau r^{\b}}d\mu_M = o(r^2),\quad \forall\, \tau \gg 1\smallskip.
 \end{equation}
\end{enumerate}
\end{corollary}

\begin{remark}
When $N= \ss^{n-1}$, and hence $\C^n_{Eu}(N)= \rr^n\setminus\{0\}$, case a) with $\b=4/3$ includes a well-known result by Meshkov, \cite{Me1}. On the other hand, case b)  with $\gamma = 4/3$, and hence a quadratic decay of $q_1$, is related to a celebrated sharp result due independently to Meshkov, \cite{Me2}, and Pan-Wolff, \cite{PW} (see also \cite{Su2} by Sun for an interesting generalization to cylinders over Zoll manifolds) which achieved unique continuation for solutions decaying more than polynomially. Their sharp decays require different techniques which work on very special spectral geometries. We will investigate this direction in a subsequent work.  
Conversely, the general method we presented in this paper, although not sharp in specific settings, permits to deal with more general geometries. As a matter of fact, Meshkov-Pan-Wolff result requires the constant $C_2$ to be small (also when $C_1=0$), while the value of $C_2$ seems not to play a role in our approach. 
\end{remark}

The proof of Corollary \ref{cor: concEucl} and of all the other results of Section \ref{section:concreteexamples} is based on a direct applications of Theorem \ref{th-meshkov} and Corollary \ref{coro_meshkov}: we exhibit a possible choice of the functions $\sigma, h_\tau, G$ that leads to the desired unique continuation result. A straightforward computation permits to check that the assumptions are satisfied in these settings and allow to apply our main theorem or even its corollary. For the reader convenience, we add some details about the relevant quantities appearing in the computations.\smallskip

Keeping the notation introduced in the statement of Theorem \ref{th-meshkov}, we set 
\[
k^L_{2,\tau}:=2F_\tau^{2} - 2(n-1) \frac{F_\tau\s'}{ \s} + F_\tau' + F_\tau G' 
\]
and 
\[
k^R_{2,\tau}:=-F_\tau'-F_\tau G'+2\frac{F_\tau\s'}{\s},
\]
so that the condition \eqref{a:k2} reads 
 \[
0\le k_{2,\tau} \le 2\min\left\{k^L_{2,\tau} \ ;\ k^R_{2,\tau}\right\}.
\]

\begin{proof}
a) We choose 
\[
\begin{cases}
 \s(r) = r\\
 h_\tau(r) = 2\tau r^{\beta}\\
G(r) = (3-2\beta)\log(r)
\end{cases}
\]
which give $F_\tau(r)=\beta\tau r^{\beta-1}+\frac{n-4+2\beta}{2r}$. We obtain $k_{2,\tau}^L=2\beta^2\tau^2r^{2\beta-2}+\tau o(r^{2\beta-2})$ and $k_{2,\tau}^R=\tau\beta^2r^{\beta-2}+o(r^{\beta-2})$, so that for $r_0$ sufficiently big and $\tau\gg 1$ (depending on $r_0$) $k_{2,\tau}:=\tau\beta^2r^{\beta-2}$ is admissible, and with this choice the constraint on $k_1$ reads $0\le k_1\le \tau^3\beta^4r^{3\beta-4}+\tau^2o(r^{3\beta-4})$, leading to the result.

b) We choose
\[
\begin{cases}
\s(r) = r\\
h_\tau(r) = \tau (\log{r})^\gamma\\
G(r)= 3\log{r}-(2\gamma-2)\log(\log{r})
\end{cases}
\]
which give $F_\tau(r)=\frac{1}{2r}\left(\tau \gamma(\log{r})^{\gamma-1}+n-4+\frac{(2\gamma-2)}{\log{r}}\right)$. We obtain $k_{2,\tau}^L=\frac{1}{2r^2}\tau^2 \gamma^2(\log{r})^{2\gamma-2}+\tau o(\frac{(\log{r})^{2\gamma-2}}{r^2})$ and $k_{2,\tau}^R=\frac{1}{2r^2}\tau\gamma(\gamma-1)(\log{r})^{\gamma-2}$, so that for $r_0$ sufficiently big and $\tau\gg1$ (depending on $r_0$) $k_{2,\tau}:=\tau\gamma^2r^{\gamma-2}$ is admissible, and with this choice the constraint on $k_1$ reads as $0\le k_1\le \frac{\tau^3\gamma^3(\gamma-1)(\log{r})^{3\gamma-4}}{8r^4}+\tau^2 o\left(\frac{(\log{r})^{3\gamma-4}}{8r^4}\right)$, leading to the result.

c) We make the same choices as in the case a). The only difference here is that we have to further assume \eqref{ass: notcoreucl} since the hypothesis of Corollary \ref{coro_meshkov} is not satisfied. 
\end{proof}

\subsection{Hyperbolic conical ends}
On Hyperbolic conical ends, we point out the following unique continuation results.

\begin{corollary}\label{cor: hyp}
Let $\E \subset \C^n_{Hyp}(N)$ be a Hyperbolic conical end and let $u \in C^{2}(\E)$ be a solution of
\[
|\Delta u | \leq q_{1} |u| + q_{2} |\nabla u|
\]
for some $0\leq q_{1},q_{2} \in C^{0}(\E)$. Let $C_1,C_2$ be non-negative constants. Then $u \equiv 0$ in each of the following cases.
\begin{enumerate}[a)]
\item $q_1(r,x) \equiv C_1r^{\frac{3\beta-3}{2}}$, $q_2(r,x) \equiv C_2r^{\frac{\beta-1}{2}}$, $1\le \beta\le 2$, 
\[
|u(r,x)|=\CO (e^{-\tau r^\beta}),\qquad \forall\,  \tau  \gg 1,\smallskip
\]
\item $q_1(r,x) \equiv C_1r^{\frac{3\beta-3}{2}}$, $q_2(r,x) \equiv C_2r^{\frac{\beta-1}{2}}$, $\beta> 2$, 
\[
|u(r,x)|=\CO (e^{-\tau r^\beta}),\qquad \forall\,  \tau  \gg 1,\smallskip
\]
and
\begin{equation}\label{ass: notcorhyp}
 \int_{\E(r_0)\setminus \E(2r)}|\nabla u|^2 e^{2\tau r^\b} d\mu_M = o(r^2)\,\qquad \forall\,  \tau  \gg 1.\smallskip
\end{equation}
\end{enumerate}
\end{corollary}
\begin{remark}
In case $N=\ss^{n-1}$, and hence $\C^n_{Hyp}(N)= \hh^n \setminus \{0\}$,  Corollary \ref{cor: hyp} a) with $\b=1$ is related to a result by Mazzeo, \cite{Ma}, on the whole hyperbolic space. Actually, in \cite{Ma} the unique continuation is proved on a horoball provided the function $u$, vanishing on the horosphere and extended to be $0$ outside the horoball, is in the appropriate Sobolev class. Thanks to Remark \ref{rem: mazzeo}, also this version can be achieved by applying  Theorem \ref{th-meshkov} with $N = \rr^{n-1}$, $\s(r) = \cosh(r)$, $h(r) = 2 \tau r$, $G(r)= r$. \end{remark}

\begin{proof}
a) We choose
\[
\begin{cases}
\s(r) = \sinh(r)\\
h_\tau(r) = 2\tau r^\beta\\
G(r)= r
\end{cases}
\]
which give $F_\tau(r)=\tau\b r^{\b-1} +\frac{n-1}{2} \coth r - \frac{1}{2}$. We obtain $k_{2,\tau}^L=2\tau^2\b^2r^{2\b-2}+\tau o(r^{2\b-2})$, and $k_{2,\tau}^R=\tau(2\coth{r}-1)\b r^{\b-1}+\tau o(r^{\b-1})$ so that for $r_0$ sufficiently big and $\tau\gg1$ (depending on $r_0$) $k_{2,\tau}:=\tau \b r^{\b-1}$ is admissible, and with this choice the constraint on $k_1$ reads as $0\le k_1\le \tau^3\b^3 r^{3\b-3}+\tau^3 o(r^{3\b-3})$, leading to the result.

b) We make the same choices as in the case a). The only difference here is that we have to further assume \eqref{ass: notcorhyp} since the hypothesis of Corollary \ref{coro_meshkov} is not satisfied.
\end{proof}

\subsection{Further geometries}
Our theorem also applies to geometries different from the Euclidean and the Hyperbolic ones. To give a glimpse of the possible applications, we discuss some unique continuation results in spaces with a warped cylindrical end with warping function $\s(r)=e^{r^{\beta}}$ for some $0<\beta\le 2$, where the upper bound is imposed so that the assumption \eqref{sigma lapl} is satisfied. 

\begin{corollary}\label{cor: exotgeom}
Let $(M,g)$ be an $n$-dimensional complete Riemannian manifold with a warped cylindrical end  $\E=\E^{n}_{\s,N}$, where $(N,g_N)$ is compact without boundary. Let $C_1,C_2$ be non-negative constants. 

a) Let $\sigma(r)=e^{r^{\b}}$, $0<\b<1$. Let $u \in C^{2}(\E)$ be a solution of
\[
|\Delta u | \leq q_{1} |u| 
\]
for some $0\leq q_{1}\in C^{0}(\E)$. Then $u \equiv 0$ if $q_{1}(r,x) \equiv C_1$ and \smallskip
 \[
|u(r,x)| = \CO \left(e^{-\tau r^{(4-\b)/3}}\right),\qquad \forall\,  \tau  \gg 1.\smallskip
 \]

b) Let $\sigma(r)=e^{r^{\b}}$, $0<\b\le 2$. Let   $u \in C^{2}(\E)$ be a solution of
\[
|\Delta u | \leq q_{1} |u|  + q_{2} |\nabla u|
\]
for some $0\leq q_{1}, q_{2}\in C^{0}(\E)$. Then $u \equiv 0$ if $q_{1}(r,x) \equiv C_1r^{2\b-2}$, $q_{2}(r,x) \equiv C_2r^{\b-1}$  and \smallskip
 \[
|u(r,x)| = \CO \left(e^{-\tau r^{\beta}}\right),\qquad \forall\,  \tau  \gg 1.\smallskip
 \]
\end{corollary}

\begin{proof}
a) We choose
\[
\begin{cases}
\s(r) = e^{r^{\beta}}\\
h_\tau(r) = \tau r^{\frac{4-\b}{3}}\\
G(r)= r^{\beta}
\end{cases}
\]
which give $F_\tau(r)=\frac{\tau (4-\b)}{6} r^{\frac{1-\beta}{3}}+\frac{(n-2)\beta}{2}r^{\beta-1}$. We obtain $k_{2,\tau}^L=\frac{\tau^2(4-\beta)^2}{18}r^{\frac{2-2\b}{3}}+\tau o(r^{\frac{2-2\b}{3}})$, and $k_{2,\tau}^R=\frac{\tau\beta(4-\beta)}{6}r^{\frac{2\beta-2}{3}}+o(r^{\frac{2\beta-2}{3}})$ so that for $r_0$ sufficiently big and $\tau\gg1$ (depending on $r_0$) $k_{2,\tau}:=0$ is admissible, and with this choice the constraint on $k_1$ reads as $0\leq k_1 \leq \frac{\tau^3(4-\beta)^3}{216}\beta+\tau^2 o(1)$, leading to the result.

b) We choose
\[
\begin{cases}
\s(r) = e^{r^{\beta}}\\
h_\tau(r) = \tau r^{\beta}\\
G(r)= r^{\beta}
\end{cases}
\]
which give $F_\tau(r)=\frac{\tau+n-2}{2}\beta r^{\beta-1}$. We obtain $k_{2,\tau}^L=\frac{(\tau+n-2)^2}{2}\beta^2r^{2\beta-2}+\tau \CO(2r^{2\beta-2})$, and $k_{2,\tau}^R=\frac{\tau+n-2}{2}\beta^2 r^{2\beta-2}+\tau o(r^{2\beta-2})$ so that for $r_0$ sufficiently big and $\tau\gg1$ (depending on $r_0$) $k_{2,\tau}:=\frac{\tau+n-2}{2}\beta^2 r^{2\beta-2}$ is admissible, and with this choice the constraint on $k_1$ reads as $0\le k_1\le \frac{(\tau+n-2)^3}{8}\beta^4r^{4\beta-4}+\tau^2o(r^{4\beta-4})$, leading to the result.
\end{proof}

\section{Geometric applications}
In this section we present two geometric applications of our unique continuation at infinity on Hyperbolic conical ends, Corollary \ref{cor: hyp}, in two completely different settings. The first application, related to inequalities of the form $|\Delta u| \leq C |u|$, is an asymptotic  estimate of the conformal factor in the prescribed scalar curvature problem. The second application, based on the inequality $|\Delta u| \leq q(x) |\nabla u|$, concerns minimal graphs.

\subsection{Conformal deformations}

Consider a Riemannian manifold $(M,g)$ with a hyperbolic conical end  $\E \subset \C^n_{Hyp}(N)$, $n \geq 3$, and the conformal metric
\[
\tilde g = u^{\frac{4}{n-2}}g, \]
with $C^\infty$ conformal factor $u:M\to (0,+\infty)$. The scalar curvatures of $g$ and $\tilde g$ are denoted by $S$ and $\tilde S$, respectively. When $N = \ss^{n-1}$ and, hence, $\E = \hh^{n} \setminus B_{R}$ is the exterior of a ball in the standard Hyperbolic space, some non-existence results for a  deformation $ u \geq \mathrm{const} >0$ with positive scalar curvature $\tilde S \geq \mathrm{const}>0$ are obtained in \cite{BPS} as a consequence of the Feller property and the compact support principle. On the other hand, when $\tilde S(x) <0$ (possibly off a compact set), a-priori estimates both from above and from below on the solution $u$ are widely studied; for an account on this subject we refer to \cite{MRS}.

Here, using unique continuation at infinity, we point out how a-priori decay estimates of $u$ must be satisfied  compared with the growth rate of $\tilde S(x)$. No sign assumption on $\tilde S$ is required.
\begin{theorem}\label{th:confdef}
Assume that, on the $n$-dimensional Hyperbolic conical end $\E \subset \C_{Hyp}(N)$, it holds
\[
\sup_{x \in N}| \tilde S(r,x)|  = \CO \left(e^{\a(r)}\right)
\]
for some function $\a: [0,+\infty) \to (0,+\infty) $ satisfying
\[
 \lim_{t \to +\infty} \frac{\a(t)}{t} = +\infty.
\]
Then, necessarily,
\[
 \limsup_{r \to +\infty} {e^{\frac{n-2}{4}\a(r)}}\sup_{\partial \E(r)}{u(r,x)} = +\infty.   
\] 
\end{theorem}
\begin{proof}
By the Yamabe equation we have 
\[
\Delta u-C(n)Su+C(n)\tilde S u ^{\frac{n+2}{n-2}}=0, \qquad\textrm{ where } C(n):=\frac{n-2}{4(n-1)}\,.
\] 
Thus, setting $C=C(n)$, it holds 
%From the Yamabe equation we have, on $\E$,
\begin{align*}
| \Delta u  | &\leq C   u  \left| -S - \tilde S  u ^{\frac{4}{n-2}}\right|\,.
\end{align*}
Now, the Bishop-O'Neill formulas for a warped product  $M= (r_1,r_2) \times_{ \s(r)} N$ tell us that
%$M= (a,b) \times_{ \s(r)} N$ 
\begin{align}
 \sect(\nabla r \wedge X) &= -\frac{\s''}{\s}\,,\\
 \sect(X \wedge Y) &= \frac{\sect_{N}(X\wedge Y) - (\s')^{2}}{\s^{2}}\,, \nonumber
\end{align}
for all o.n. $X,Y \in TN$. In our case, since $N$ is compact and $\s(r) = \sinh(r)$ we have that $\C_{Hyp}(N)$ has bounded sectional, hence scalar, curvature. Therefore, using the assumption on $\tilde S$,
\begin{equation*}
 | \Delta u  | \leq C u   \left(1+ |\tilde S |u^{\frac{4}{n-2}} \right) \leq C u  \left( 1 + e^{\a(r)} u^{\frac{4}{n-2}}\right).
\end{equation*}
By contradiction, suppose that
 \[
 0 < u(r,x) = \CO \left(e^{-\frac{n-2}{4}\a(r)} \right).
 \]
 Then, there exists a constant $A>0$ such that
 \[
 | \Delta u | \leq  A u,\, \text{on }\E.
 \]
Since
 \[
u(r,x) \leq C_{\tau} e^{-\tau  r},\quad \forall \tau >0,
 \]
 by the unique continuation property at infinity of Corollary \ref{cor: hyp} we conclude
 \[
0<  u(x) \equiv 0,\quad \text{on }\E,
 \]
 a contradiction.
\end{proof}

\begin{remark}
    An analogous conclusion, with the same proof, can be obtained in other geometries. For instance, using Corollary \ref{cor: exotgeom} case $a)$, the Theorem applies to manifolds with a warped cylindrical end with warping function $\sigma(r)=e^{r^\b}$, $0<\b<1$, up to requiring that $\lim_{t \to +\infty} \a(t)/t^{\frac{4-\b}{3}} = +\infty.$\end{remark}
%up to using the corresponding unique continuation result of Section \ref{section:concreteexamples}.

\subsection{Minimal graphs}

Carleman estimates, and corresponding unique continuation at infinity,  have natural applications to rigidity questions in submanifold theory, both in Riemannian and in weighted ambient spaces; see for instance \cite{Wa1, Wa2, De, Su1,Su2}. Rigidity, here, means that whenever a given end of the submanifold approaches, in the graphical sense and with a certain speed, a reference submanifold in the same category, then the end is in fact included in the reference object. This section aims at providing an instance of these rigidity phenomena. \smallskip

Given a domain $\Omega$ inside the Riemannian manifold $(M,g)$, we say that the (possibly bordered) two-sided hypersurface $\Sigma \subset M \times \rr$ is given graphically over $\Omega$ if there exists a smooth function $u: \Omega \to \rr$ such that $\Sigma= \Gamma_u(\Omega)$, where $\Gamma_u : \Omega \to M \times \rr$ is the isometric embedding $\Gamma_u(x)=(x,u(x))$. The mean curvature function $H$ of $\Sigma$ with respect to the upper pointing Gauss map is defined in terms of $u$ by the equation
\begin{equation}
    H = -\dive\left( \frac{\nabla u}{\sqrt{1+|\nabla u|^2}}\right).
\end{equation}
Here and throughout this section, unless otherwise specified, all the differential operators (gradient, Hessian, Laplacian, divergence and so on) are understood in the base space $(M,g)$.

The graphical hypersurface $\Sigma$ is minimal if $H \equiv 0$. In this case, we say that $u$ satisfies the minimal surface equation that can be written in the equivalent form
\begin{equation}
    \Delta u = \frac{\Hess(u)(\nabla u ,\nabla u)}{1+|\nabla u|^2}.
\end{equation}
In particular
\begin{equation}\label{MSE2}
    |\Delta u| \leq q\, |\nabla u|
\end{equation}
where
\begin{equation}\label{MSE3}
    q = |\Hess(u)||\nabla u|.
\end{equation}

From \eqref{MSE2} and \eqref{MSE3} it is clearly visible that, once the Hessian of the graphical function $u$ is bounded and a certain decay rate of $u$ propagates to its gradient, the unique continuation results developed in Section \ref{section:concreteexamples} apply and give corresponding rigidity conclusions.
Accordingly, the crucial step is represented by the following (possibly well known) result.
In its statement we shall implicitly assume the following very classical fact, \cite{JK, He}:\smallskip

On a complete Riemannian $(M,g)$ with positive injectivity radius $r_{\inj}(M)=i>0$ and bounded sectional curvature $|\sect| \leq S$ the $C^{1,\a}$-harmonic radius of $M$ with ``precision'' $Q>1$, denoted by $r_{\harm}(M)$, satisfies the estimate
\[
r_{\harm}(M) \geq 4r_{0}
\]
for some constant $r_{0}=r_{0}(i,S,\a,Q)>0$. Thus, if we let $g_{ij}$ and $g^{ij}$ to denote the coefficients of the metric tensor and those of its inverse, we have that, in harmonic coordinates within any ball $B_{r_{0}}(p)$, the following conditions are satisfied: (a) $Q^{-1} \d_{ij} \leq g_{ij} \leq Q \d_{ij}$ in the sense of quadratic forms; (b) both $g_{ij}$ and $g^{ij}$ are $C^{1,\a}$-uniformly bounded. Since the actual value of $Q$ is irrelevant in our setting, we shall always take $Q=2$.

\begin{theorem}\label{Th:mimalgraphestimate}
    Let $(M,g)$ be a complete Riemannian manifold of dimension $2 \leq n \leq 5$, with $C^1$-bounded geometry. Namely, $M$ has positive injectivity radius $r_{\inj}(M)=i>0$, bounded sectional curvature  $|\sect| \leq S$ and bounded covariant derivative of the Riemann tensor $|D\riem| \leq c$.\smallskip
    
    \noindent Let $\Sigma = \Gamma_u (\Omega)$ be a vertical minimal graph over $\Omega \subseteq M$ with $|u|\le L$ in $\Omega$. Then, for any $p \in \Omega$ and $0<4R<\min(\dist(p,\partial \Omega),r_{\harm}(M))$ there exists a constant $C=C(R,i,S,c,L)>0$ such that  $\| u \|_{C^{2}(B_{2R}(p))}\leq C$ and $u$ satisfies \eqref{MSE2} on $B_{2R}(p)$ with
\begin{equation}
0 \leq q(p) \leq C \sup_{B_{R}(p)}|u|.
\end{equation}
\end{theorem}

\begin{proof}
During the proof, each time the constant $C$ appears it is understood that it depends only on the parameters in the statement of the Theorem.\smallskip

 First of all, by  Spruck interior gradient estimates, \cite[Theorem 1.1]{Spruck}, it holds
\begin{equation}\label{gradientest}
\sup_{B_{R}(y)} |\nabla u | \leq C,\quad \forall y \in B_{2R}(p).
\end{equation}
Next, we show that $|\Hess(u)|$ is uniformly bounded. To this end, thanks to \eqref{gradientest}, it is enough to prove that $\Sigma$ has uniformly bounded second fundamental form $A$, because
\begin{equation}
 \Hess(u) =\pm \sqrt{1+|\nabla u|^{2}} \, A( d\Gamma_{u} , d\Gamma_{u}).
\end{equation}
But this essentially follows from  classical work by Schoen-Simon-Yau, \cite{SSY}. Indeed, $\Sigma$ is a minimal hypersurface inside the complete Riemannian manifold $(M\times \rr, g + dt^2)$ that inherits from $M$ the same $C^1$-bounds $i,S,c$ on its geometry.  Since $\Sigma= \Gamma_{u}(\Omega)$ is graphical with bounded slope $|\nabla u | \leq C$ and $B^{\Sigma}_{2\ve}(\Gamma_{u}(p)) \subseteq B_{2\ve}(p) \times \rr$, we have that 
\begin{equation}\label{volumeestimate}
 \vol(B^{\Sigma}_{2\ve}(\Gamma_{u}(p)) \leq \int_{B_{2\ve}(p)} \sqrt{1+|\nabla u|^{2}} \, d\mu_{M}\leq C \,  \vol (B_{2\ve}(p)) \leq C  \, \vol (B_{2\ve}^{\hh^{n}_{-S}}).
\end{equation}
 Thus, if $0<\ve< \tilde C$ for a suitable constant $0<\tilde C = \tilde C(S,i)<R$, then the right hand side of \eqref{volumeestimate} can be made so small that $\Sigma$ enjoys the Euclidean isoperimetric inequality by Hoffman-Spruck, \cite{HS}, at the scale $2\ve$. We can therefore adapt the argument in \cite[Theorem 3]{SSY} (see Lemma \ref{lemma:SSY} below) to deduce that, up to enlarging $C$ (but with the same dependence on the parameters),
\begin{equation}
 \sup_{B_{\ve}(y)} |A| \leq C,\quad \forall y \in B_{2R}(p).
\end{equation}
We have thus obtained that
\begin{equation}
\| u \|_{C^{2}(B_{2R}(p))} \leq C.  
\end{equation}
It follows that, letting
\[
a = (1+|\nabla u|^{2})^{-1/2},
\]   
the linear operator in divergence form
\[
\CL w = \div (a \nabla w) = \frac{1}{\sqrt{\det g}} \partial_{i}\left( a \, \sqrt{\det g} \, g^{ik} \,\partial_{k}w \right)
\]
is uniformly elliptic with  uniformly $C^{0,\a}$-bounded coefficients. Since $u$ solves $\CL u = 0$ on $B_{2R}(p)$, we can apply $C^{1}$-Schauder estimates, \cite[Theorem 2.28]{FR}, and deduce that
\begin{equation}
\sup_{B_{R/2}(p)} |\nabla u | \leq C^{\ast} \sup_{B_{R}(p)} |u|
\end{equation}
where $C^{\ast}$ depends only on $n$ and on $\| u \|_{C^{2}(B_{2R}(p))}$, hence only on $C$.
\end{proof}

For the sake of completeness, we  outline the proof of the needed version of \cite[Theorem 3]{SSY}.

\begin{lemma}\label{lemma:SSY}
 Let $(N,g)$ be an $n$-dimensional complete manifold, $2\leq n \leq 5$, with $C^{1}$-bounded curvature: $|\sect|  \leq S$ and $|D \riem| \leq c$. Let $B^{\Sigma}_{2R}(p) \hookrightarrow N$ be a relatively compact embedded minimal $n$-disk with $2R < r^{N}_{\inj}(p)$. Let $0 < \ve \ll 1$ be so small that, for every $q \in B_{R}^{\Sigma}(p)$, $B^{\Sigma}_{2\ve}(q)$ enjoys the Hoffman-Spruck Euclidean isoperimetric inequality. Then, there exists a constant $C= C(R,S,c, \ve)>0$ such that
 \[
\sup_{B^{\Sigma}_{\ve}(q)} |A| \leq C, \quad \forall q \in B_{R}^{\Sigma}(p).
 \]
\end{lemma}

\begin{proof}
From (1.27) in \cite{SSY} with $K_{1}=S$, $K_{2}= -S$ we know that
 \[
 \Delta |A|^{2} \geq -4c|A| - 6n S |A|^{2}- 2|A|^{4}.
\]
 Fix a constant $\b>0$ whose value will be specified later and let $u = \b^{2}+ |A|^{2}$.  Then,
\begin{align*}
 \Delta u &\geq -4c|A| - 6n S |A|^{2} - 2|A|^{4} \\
 &= -(6nS+2 |A|^{2}) u + 2 (\b |A| - c\b^{-1})^{2} + (6 nS \b^{2} - 2c^{2}  \b^{-2})\\
 &\geq -(6nS+2 |A|^{2}) u 
\end{align*}
provided $\b= \b(S,n,c) \gg1$ is so large that $6nS \b^{2} - 2c^{2}  \b^{-2} \geq 0$. We have thus shown that $u \geq 0$ is a smooth solution of
\[
\Delta u + f u \geq 0,\quad f = 6nS+2 |A|^{2}. 
\]
Now, thanks to the Euclidean isoperimetric inequality, which implies an Euclidean lower volume bound, in $B_{\ve}^{\Sigma}(p)$ we have the validity of a scale-invariant Sobolev inequality. Therefore, we can apply Euclidean methods to get a mean value inequality of the form (see e.g. \cite[Lemma 19.1]{Li})
\[
\sup_{B^{\Sigma}_{\ve}(q)}  |A|^{2} \leq \sup_{B^{\Sigma}_{\ve}(q)} u \leq \frac{D}{\vol B^{\Sigma}_{3\ve/2}(q) }\int_{{B^{\Sigma}_{3\ve/2}(q)} } u = D\b^{2} + \frac{D}{\vol B^{\Sigma}_{3\ve/2}(q) }\int_{{B^{\Sigma}_{3\ve/2}(q)} }  |A|^{2}
\]
where, for any fixed $t > n/2$, the constant $D>0$ depends on the average $L^{2t}$-norm of $f$. Since volumes at the scale $2\ve$ are Euclidean, the proof now can be completed exactly as in \cite[Theorem 3]{SSY}.
\end{proof}

 Note that, if $\Omega = B_{\d}(D)$ is a $\d$-uniform tubular neighborhood of a domain $D\subseteq M$, then the radius $R$ in Theorem \ref{Th:mimalgraphestimate} can be chosen uniformly, independent of the center $p \in D$.
In this setting, interesting examples, where the asymptotic decay of $u$ is inherited by $q$, are the following.

\begin{example}[halfspaces or halfcylinders]
Let $(N,g)$ be a complete manifold with $C^1$-bounded geometry and let $u: (R,+\infty) \times N \to \rr$  define a bounded minimal graph. Given any $f(r) \searrow 0^{+}$, if $\sup_{x \in N}| u(r,x) | \leq f(r)$ for all $r>R$ then $\sup_{x \in N} | q (r,x) |  \leq  C\, f(r-1) $ for $r> R+1$. 
\end{example}

\begin{example}[exterior domains]
 Let $(M,g)$ be a complete manifold with $C^1$-bounded geometry and let $u: M \setminus \bar B_{R}(o) \to \rr$ define a bounded minimal graph. Given any $f(r) \searrow 0^{+}$, if  $\sup_{\partial B_{r}(o)}| u | \leq  f(r)$ for all $r>R$ then $\sup_{\partial B_{r}(o)} | q  | \leq  C\,f(r-1)$ for  $r >R+1$.
\end{example}

\begin{example}[conical ends]\label{example:minimal-conical}
Let $\E(R)$ be an end of either $\C_{Eu}(N)$ or $\C_{Hyp}(N)$ and let $u(r,x): \E(R) \to \rr$ define a bounded minimal graph. Given any $f(r) \searrow 0^{+}$, if $\sup_{x \in N}| u(r,x) | \leq f(r)$ for all $r>R$ then $\sup_{x \in M} | q (r,x) |  \leq  C f(r-1) $ for $r>R+1$.\smallskip

\noindent Indeed, we have already observed that, as a consequence of Bishop-O'Neill, the cone has bounded sectional curvature. The covariant derivative of Riemann is also  bounded. To see this, for a warped product $(M^n,g) = ((r_1,r_2) \times N , dr\otimes dr + \s^2(r) g_N)$, 
%generic warped product} $(M^n,g) = ((a,b) \times N^{n-1} , dr\otimes dr + \s^2(r) g_N)$, 
we fix the index convention $i,j,k,t,s=1,\cdots,n-1$ and, according to Bishop-O'Neill, we start by writing in local coordinates $x_0 = r,x_1,\cdots,x_{n-1}$ the nonzero Christoffel symbols:

\begin{align*}
    \Gamma^k_{ij} &= (\Gamma^N)^k_{ij}\\
    \Gamma^0_{ij} &= -\frac{\s'}{\s}g_{ij}\\
    \Gamma^k_{0,j} &= \frac{\s'}{\s} \d_{kj}
\end{align*}
and the nonzero components of the Riemann tensor:
\begin{align*}
    R_{i0j0} &= - \frac{\s''}{\s} \frac{1}{2}(g \KN g)_{i0j0} \\
    R_{ijkt} &= R^N_{ijkt} + \left(\frac{\s'}{\s}\right)^2 \frac{1}{2}(g \KN g)_{ijkt}
\end{align*}
where $\KN$ denotes the Kulkarni-Nomizu product. 
Recall that the covariant derivatives of the Riemann tensor write, for $a,b,c,d,e,...=0,\cdots,n-1,$
\[
D_aR_{bcde}=\partial_{x_a}R_{bcde} - \Gamma_{ab}^fR_{fcde}- \Gamma_{ac}^fR_{bfde}- \Gamma_{ad}^fR_{bcfe}- \Gamma_{ae}^fR_{bcdf},
\]
while its tensorial norm is given by 
\[
|D\riem|^2=D_aR_{bcde}\,D_{\bar a} R_{\bar b\bar c \bar d\bar e}g^{a\bar a}g^{b\bar b}g^{c\bar c}g^{d\bar d}g^{e\bar e}.
\]
Since the Riemann tensor of $M$, the Christoffel symbols of $M$, the derivatives of the Riemann tensor of $N$ and the quantities $(\sigma'/\sigma)$ and $(\sigma'/\sigma)'$  are all bounded, a straightforward computation shows that $|D\riem|$ is bounded in $\E(R)$, as claimed.\\
Now, we can think of $\E(R)$ as an end of the Riemannian double $\D(\E(R/2))$ which is a complete Riemannian manifold that still has $C^1$-bounded curvature. It is not difficult to see that the injectivity radius of $M=\D(\E(R/2))$ is positive. For instance, one can note that $\vol (B^M_1)$ does not collapse at infinity and apply a classical result by Cheeger-Gromov-Taylor, \cite{CGT}.
\end{example}

According to Example \ref{example:minimal-conical}, we have the following direct geometric consequence of Corollary \ref{cor: hyp}. 

\begin{theorem}\label{th:hyperbolicminimalgraph}

    Let $\Sigma = \Gamma_u(\E)\subset \C^n_{Hyp}(N)\times \rr$, $2 \leq n\leq 5$, be a vertical minimal graph over the Hyperbolic conical end $\E$ and let $H\in\rr$. If $\Sigma$ converges to a totally geodesic slice $\C^n_{Hyp}(N) \times \{H\}$ more than exponentially, then
    \[
    \Sigma \subseteq \C^n_{Hyp}(N) \times \{H\}.
    \]
\end{theorem}

\begin{remark}
    A catenoid-like hypersurface in $\hh^n \times \rr$, e.g. \cite{IP}, is the graphical hypersurface $\Sigma= \Gamma_u(\hh^n\setminus B_1(o))$ defined, in polar coordinates of $\hh^n$ around $o$, by the function
    \[
    u(r,x) = \int_1^{r} \frac{\sinh^{n-1}(1)}{\sqrt{\sinh^{2(n-1)}(s)-\sinh^{2(n-1)}(1)}}ds.
    \]
    Since $\Sigma$ is a non-trivial minimal graph converging exponentially to a totally geodesic slice, Theorem \ref{th:hyperbolicminimalgraph} is essentially sharp.
\end{remark}

\begin{remark}
    The same result can be obtained if we replace the hyperbolic end with a warped product end of a Cartan-Hadamard manifold with pinched negative curvature. 
\end{remark}

\medskip

\subsection*{Acknowledgements.} We would like to thank the anonymous referee for corrections and suggestions which improved the presentation of this paper. All the authors are members of the INdAM-GNAMPA groups. The first author is supported by the INdAM-GNAMPA project “Mancanza di regolarità e spazi non lisci: studio di autofunzioni e autovalori”, CUP E53C23001670001.

\end{document}